\documentclass[]{interact}

\usepackage{amssymb}
\usepackage{amsmath}
\usepackage{amsthm}
\DeclareMathAlphabet\mathbfcal{OMS}{cmsy}{b}{n}
\usepackage{algorithmicx}
\usepackage{algorithm}
\usepackage{algpseudocode}
\usepackage{float}
\usepackage{subcaption}
\newfloat{algorithm}{tb}{alg}
\usepackage{graphicx}
\usepackage{hyperref}
\usepackage{cleveref}
\usepackage{tikz, pgfplots}
\pgfplotsset{compat=1.17}
\usepackage{tikz-qtree}
\usetikzlibrary{arrows,automata,matrix,positioning,patterns}

\theoremstyle{plain}
\newtheorem{theorem}{Theorem}[section]
\newtheorem{lemma}[theorem]{Lemma}

\newtheorem{proposition}[theorem]{Proposition}

\theoremstyle{definition}

\newtheorem{example}[theorem]{Example}

\theoremstyle{remark}

\newtheorem{assumption}{Assumption}

\usepackage[numbers,sort&compress]{natbib}
\bibpunct[, ]{[}{]}{,}{n}{,}{,}

\usepackage{tabularx}
\usepackage{booktabs}
\newcolumntype{j}{>{\centering\arraybackslash}p}
\newcolumntype{z}{>{\raggedleft\arraybackslash}p}
\newcolumntype{J}{>{\centering\arraybackslash}X}
\usepackage{multirow}

\usepackage{xcolor}

\definecolor{Color0}{HTML}{E41A1C}
\definecolor{Color1}{HTML}{377EB8}
\definecolor{Color2}{HTML}{4DAF4A}
\definecolor{Color3}{HTML}{984EA3}
\definecolor{Color4}{HTML}{FF7F00}
\usepackage{comment}

\DeclareMathOperator*{\myint}{int}
\DeclareMathOperator*{\argmin}{arg\,min}

\newcommand{\dif}{\mathrm{d}}
\newcommand{\dcocpp}{{\ttfamily dco\kern-.08em{\raisebox{-.1ex}{/}\kern-.15em {c\kern-.03em{\raisebox{-.18ex}{+\kern-.028em{+}}}}}}}

\begin{document}

\title{Numerical simulation of differential-algebraic equations with embedded global optimization criteria}
\author{
\name{J. Deussen\textsuperscript{a}\thanks{Informatik 12: Software and Tools for Computational Engineering, RWTH Aachen University, Germany.}} \\
\and
Jonathan H\"user\footnotemark[1]
\and
Uwe Naumann\footnotemark[1]}
\author{
\name{Jens Deussen\thanks{CONTACT Jens Deussen. Email: deussen@stce.rwth-aachen.de}
and Jonathan Hüser and Uwe Naumann}
\affil{Informatik 12: Software and Tools for Computational Engineering, RWTH Aachen University, Germany}
}
\date{}

\maketitle

\begin{abstract} \small\baselineskip=9pt
We are considering differential-algebraic equations with embedded optimization
criteria (DAEOs) in which the embedded optimization problem is solved by global
optimization.
This actually leads to differential inclusions for cases in which there are
multiple global optimizer at the same time.
Jump events from one global optimum to another result in nonsmooth DAEs
and thus reduction of the convergence order of the numerical integrator to
first-order.
Implementation of event detection and location as introduced in this work
preserves the higher-order convergence behavior of the integrator.
This allows to compute discrete tangents and adjoint sensitivities for optimal
control problems.

\end{abstract}

\section{Introduction}
Solving dynamic models that are described by underdetermined systems of
differential-algebraic equations (DAEs), i.e., where there are more
algebraic states than algebraic equations, often requires to embed an
optimization criterion to find a unique solution.
These models are called DAEs with embedded optimization criteria (DAEOs).

Prime application for DAEOs can be found in process system engineering.
Separation processes with thermodynamic equilibrium between phases
are often modeled with DAEOs.
The DAE describes the dynamic behavior and a nonlinear programs represents
the phase equilibrium which is at the minimum of the Gibbs free energy
\cite{Baker1982,Gopal1999,Sahlodin2016}.

For simulation \cite{Ploch2020, Ploch2022}
and optimization \cite{Cao2003, Hjersted2006, Raghunathan2004}
of DAEOs the nonlinear programs are often substituted with
first-order optimality conditions, i.e., Karush-Kuhn-Tucker (KKT) conditions.
The local approach only guaranties exact solution for convex nonlinear programs.
For nonconvex nonlinear programs, the solution point might only be locally optimal.
Furthermore, substitution of the optimization problem by the KKT conditions yields
a nonsmooth DAE due to switching events, i.e., changes of the active set.
This kind of problem is solved either with a simultaneous or with a sequential
approach \cite{Biegler2010}.
The simultaneous approach results in solving a single large-scale NLP.
In contrast, the sequential approach requires interaction of the numerical
integrator and the numerical optimizer.

In this paper, we propose the simulation of DAEOs with deterministic global
optimization (DGO) methods instead of substituting the optimization problem
with first-order optimality conditions.
This has the advantage that the solution point of the optimization problem
is guaranteed to be a global optimum.
Embedding a global optimization problem requires to consider differential
inclusions as a generalization of differential equations for the cases where
several global optimizer exist.
Among other things, this is also the case when the global optimizer jumps.
A jump of the global optimizer is referred to as an event, which can occur
due to the nature of dynamic systems.
In the presence of an event the resulting DAE system becomes nonsmooth.
Using the local optimum of the previous time step even if the time period
covers an event amounts to a type of discontinuity locking \cite{Park1996}.
Unfortunately, the convergence order of the integrator is reduced to
first-order without explicit treatment of event locations.
To achieve second-order convergence across jumps of the global optimum it is
necessary to implement an adaptive time stepping or event location procedure.
We will describe how to detect and locate events.

An alternative approach for obtaining a higher convergence for the integrator
is presented in \cite{Griewank2018}.
They utilize a generalization of algorithmic differentiation~(AD) to treat
nonsmooth right hand sides of ordinary differential equations (ODEs).
Further information on the theory of AD can be found in
\cite{Griewank2008, Naumann2012} and for information on nonsmooth AD it is
referred to \cite{Griewank2013, Hueser2022}. 

The paper is organized as follows:
In \Cref{sec:TB}, we describe the mathematical formulation of the type of
problem we are considering in this paper.
This includes the description of the DAEO, and assumptions on the DAEO as well
as on the events, i.e., jumps of the global minimizer.
\Cref{sec:GO} gives an overview of interval computations and deterministic
global optimization.
Based on \cite{Deussen2019}, we show how to obtain all convex subdomains of a
nonconvex objective function that potentially contain a local minimum.
The numerical simulation of the DAEO is described in \Cref{sec:NS}.
This includes a description on how to detect whether an event has happened on a
time period (event detection) and how to find the time step at which the event
takes place (event location).
In \Cref{sec:NE} the presented methods are applied to two example functions.
For the first example, we derive an analytical solution to examine the
convergence behavior of the simulation with and without explicit treatment of
the events.
\Cref{sec:C} summarizes the results and gives an outlook on future work.

\section{Theoretical Background}
\label{sec:TB}

We consider the initial value problem for a differential inclusion
with an embedded global optimization problem
\begin{align}
\begin{split}
  x(t_0) &= x_0 \\
  x'(t) &= f(x(t), y(t)) \\
  y(t) &\in \argmin_y h(x(t), y(t)) \quad \forall t \in (t_0, T) \; ,
\end{split}
\label{eq:prob}
\end{align}
where $f:\mathbb R^{n_x}\times\mathbb R^{n_y}$ is the differential part of the
system and $h:\mathbb R^{n_x}\times\mathbb R^{n_y}$ is the objective function.
We use the notation $x'= \partial x/\partial t$ for the derivative with respect
to time, $\partial_x=\partial /\partial x$ for the (partial) derivative and
$\dif_x = \dif / \dif x$ for the total derivative with respect to $x$ where it
is crucial to distinguish from the partial derivative in the context of
implicit differentiation.
Second partial derivatives are denoted by
$\partial^2_{yx}=\partial^2 / \partial y\partial x$.
The problem in \cref{eq:prob} is a differential inclusion instead of an
differential equation problem because the optimum of the embedded
optimization problem is not necessarily unique.
In this paper we consider the case where the solution set consists of
a finite set of isolated strict local optima that are also potential global optima
\begin{equation*}
  \argmin_y h(x, y) = \bigcup_{i = 1}^S \{ y^i \}\;,
\end{equation*}
with $h(y^i) = h(y^j)$ for all $i, j \in \{1, \dots, S\}$ and
$h(y^i) < h(z)$ for all
$i \in \{1, \dots, S\}$, $z \in \mathbb R^{n_y} \setminus \argmin_y h(x, y)$.
Notably, that also means that the implicit set-valued map $y(t) = y(x(t))$ is
not convex.

\begin{assumption}\label{assump:Lipschitz}
  The functions $f$ and $h$ are twice Lipschitz continuously differentiable
  with respect to $(x, y)$.
\end{assumption}
For each strict local optimum we have necessary and sufficient local optimality
conditions
\begin{align*}
  0 &= \partial_y h(x, y) \\
  0 &\prec \partial^2_{yy} h (x, y) \; .
\end{align*}
The gradient with regard to $y$ is zero and the Hessian is positive definite.
\begin{assumption}
  The initial value problem is posed in such a way that the necessary and
  sufficient local optimality conditions hold for $(x(t), y^i(t))$ for all
  $i \in \{1, \dots, S\}, t \in [t_0, T]$.
\end{assumption}
In order to turn the differential inclusion into a discontinuous differential
equation with a unique solution we consider the case where the solution set has
size larger than one only for a set of times that has measure zero.
We make the even stronger assumptions that the set of times where the solution
set is larger than one is finite in order to make handling these events
numerically feasible.
\begin{assumption} \label{ref:assump}
  There exists a finite set of events $0 < t^{e_1} < \dots < t^{e_K} < T$ with
  \[\left| \argmin_y h(x(t^{e_j}), y(t^{e_j})) \right| = 2\;,\]
  for all $j \in \{1, \dots K\}$ and we have
  \[\left| \argmin_y h(x(t), y(t)) \right| = 1\;,\]
  for all $t \in [t_0, T]$ with $t \not= t^{e_j}$
  for some $j \in \{1, \dots K\}$.
\end{assumption}
With the previous assumption we can have two cases for each $t^{e_j}$.

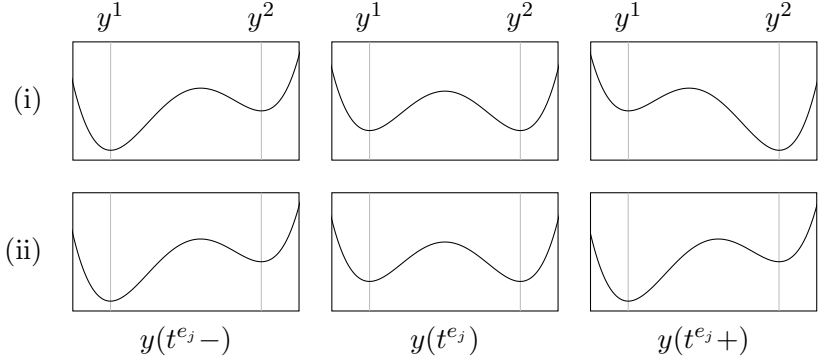
\begin{figure*}[tb]
  \centering
\begin{tikzpicture}[declare function=
  {h(\x,\t)= (1.0-\x*\x)*(1.0-\x*\x) - (\t-0.5)*sin(deg(\x*pi/2.0))+0.75;} ]
\begin{axis}[width=0.32\linewidth, height=0.22\linewidth,
xmin=-1.5, xmax=1.5, ymin=0, ymax=3, ylabel={(i)},
ylabel style={rotate=-90},
xticklabels={}, extra x tick labels={$y^1$,$y^2$}, extra x ticks={-1,1},
extra x tick style={grid=major,ticklabel pos=upper},
tick style={draw=none}, yticklabels={}]
\addplot[color=black, domain=-2:2, samples=200, smooth] (x, {h(x,0)});
\end{axis}
\begin{scope}[xshift=0.24\linewidth]
\begin{axis}[width=0.32\linewidth, height=0.22\linewidth,
xmin=-1.5, xmax=1.5, ymin=0, ymax=3,
xticklabels={}, extra x tick labels={$y^1$,$y^2$}, extra x ticks={-1,1},
extra x tick style={grid=major,ticklabel pos=upper},
tick style={draw=none}, yticklabels={}]
\addplot[color=black, domain=-2:2, samples=200, smooth] (x, {h(x,0.5)});
\end{axis}
\end{scope}
\begin{scope}[xshift=0.48\linewidth]
\begin{axis}[width=0.32\linewidth, height=0.22\linewidth,
xmin=-1.5, xmax=1.5, ymin=0, ymax=3,
xticklabels={}, extra x tick labels={$y^1$,$y^2$}, extra x ticks={-1,1},
extra x tick style={grid=major,ticklabel pos=upper},
tick style={draw=none}, yticklabels={}]
\addplot[color=black, domain=-2:2, samples=200, smooth] (x, {h(x,1)});
\end{axis}
\end{scope}
\begin{scope}[yshift=-0.14\linewidth]
\begin{axis}[width=0.32\linewidth, height=0.22\linewidth,
xmin=-1.5, xmax=1.5, xlabel=$y(t^{e_j}-)$,
x label style={yshift=0.2cm},
ymin=0, ymax=3, ylabel={(ii)},
ylabel style={rotate=-90},
xticklabels={}, extra x tick labels={}, extra x ticks={-1,1},
extra x tick style={grid=major},
tick style={draw=none}, yticklabels={}
]
\addplot[color=black, domain=-2:2, samples=200, smooth] (x, {h(x,0)});
\end{axis}
\end{scope}
\begin{scope}[xshift=0.24\linewidth,yshift=-0.14\linewidth]
\begin{axis}[width=0.32\linewidth, height=0.22\linewidth,
xmin=-1.5, xmax=1.5, xlabel=$y(t^{e_j})$,
x label style={yshift=0.2cm},
ymin=0, ymax=3,
xticklabels={}, extra x tick labels={}, extra x ticks={-1,1},
extra x tick style={grid=major},
tick style={draw=none}, yticklabels={}]
\addplot[color=black, domain=-2:2, samples=200, smooth] (x, {h(x,0.5)});
\end{axis}
\end{scope}
\begin{scope}[xshift=0.48\linewidth,yshift=-0.14\linewidth]
\begin{axis}[width=0.32\linewidth, height=0.22\linewidth,
xmin=-1.5, xmax=1.5, xlabel=$y(t^{e_j}+)$,
x label style={yshift=0.2cm},
ymin=0, ymax=3,
xticklabels={}, extra x tick labels={}, extra x ticks={-1,1},
extra x tick style={grid=major},
tick style={draw=none}, yticklabels={}]
\addplot[color=black, samples=200, smooth] (x, {h(x,0)});
\end{axis}
\end{scope}
\end{tikzpicture}
\caption{Cases resulting from \Cref{ref:assump}.}
\label{fig:assump3}
\end{figure*}
\begin{itemize}
  \item[(i)] The switch from one global optimize to another:
    $$\argmin_y h(x(t^{e_j}), y(t^{e_j})) = \{ y^1(t^{e_j}), y^2(t^{e_j}) \}\;,$$
    and
    \begin{align*}
      \lim_{t \rightarrow t^{e_j}-} \argmin_y h(x(t), y(t)) &= \{ y^1(t^{e_j}) \}\;, \\
      \lim_{t \rightarrow t^{e_j}+} \argmin_y h(x(t), y(t)) &= \{ y^2(t^{e_j}) \} \; .
    \end{align*}
    See \Cref{fig:assump3} (top) for a visualization.

  \item[(ii)] The touching of a second local optimum:
    $$\argmin_y h(x(t^{e_j}), y(t^{e_j})) = \{ y^1(t^{e_j}), y^2(t^{e_j}) \}\;,$$
    and
    \begin{align*}
      \lim_{t \rightarrow t^{e_j}-} \argmin_y h(x(t), y(t)) &= \{ y^1(t^{e_j}) \} \;,\\
      \lim_{t \rightarrow t^{e_j}+} \argmin_y h(x(t), y(t)) &= \{ y^1(t^{e_j}) \} \; .
    \end{align*}
    See \Cref{fig:assump3} (bottom) for a visualization.
\end{itemize}
The second case is not relevant for the solution of the differential equation
because it occurs on a set of measure zero.
We make a transversality assumption that precludes the second case.
First, we define the condition of the touching two local optima $y^1$ and $y^2$
as the root of the event function $H(x(t))$
\begin{equation*}
  0 = h(x(t), y^1(x(t))) - h(x(t), y^2(x(t))) = H(t) \; .
\end{equation*}
\begin{assumption}
  All events are transversal:
  Assuming that $y(x(t^{e_j})) = y^1(x(t^{e_j}-))$ we have
  \begin{align*}
    \partial_x H(x(t^{e_j})) f(x(t), y^1(t^{e_j})) &> 0 \\
    \partial_x H(x(t^{e_j})) f(x(t), y^2(t^{e_j})) &> 0 \quad j \in \{1, \dots, K\} \; .
  \end{align*}
\end{assumption}
With the above assumptions we can write the original differential inclusion as in
\cref{eq:prob}
as multiple initial value problems for differential equations for
$j = 0, \dots, K+1$:
\begin{align}
\begin{split}
  x^j(t^{e_j}) &= \begin{cases}
    x_0 & \text{if } j=0\\
    x^{j-1}(t^{e_j}) & \text{otherwise}
  \end{cases} \\
  (x^j)'(t) &= f(x^j(t), y^j(t)) \\
  \{ y^j(t) \} &= \argmin_y h(x^j(t), y^j(t))
  \quad t \in (t^{e_j}, t^{e_{j+1}}) \; ,
  \end{split}
  \label{eq:phases}
\end{align}
where $t^{0} = t_0$ and $t^{K+1} = T$.
The time periods $\left(t^{e_j},t^{e_{j+1}}\right)$, $j=0,\ldots,K$,
will be referred to as phases.

\begin{proposition}
  \Cref{eq:phases} is a DAEO for each phase because $y^j(t)$ is a
  locally unique implicit function of $x(t)$ that is also Lipschitz continuous
  in $x(t)$.
\end{proposition}
\begin{proof}
  The first-order optimality condition that we assumed to hold for all local
  optima is
  \begin{equation*}
    0 = \partial_y h(x, y) \; .
  \end{equation*}
  By the implicit function theorem and using the positive definiteness of the
  Hessian we have
  \begin{align*}
    \dif_x \partial_y h(x, y) &= \partial^2_{yy}h(x, y) \partial_x y(x)
    + \partial^2_{yx}h(x, y) = 0 \\
    \Rightarrow \quad \partial_x y(x)
    &= \partial^2_{yy}h(x, y)^{-1} \partial^2_{yx} h(x, y) \; .
  \end{align*}
\end{proof}

Because $y$ is only given implicitly, the above problems are really DAEs.
We get the DAE formulation by linearizing the first-order optimality condition:
\begin{align}
\begin{split}
  x^j(t^{e_j}) &= \begin{cases}
    x_0 & \text{if } j=0\\
    x^{j-1}(t^{e_j}) & \text{otherwise}
  \end{cases} \\
  (x^j)'(t) &= f(x^j(t), y^j(t)) \\
  0 &= \partial_y h(x^j(t), y^j(t)) \quad t \in (t^{e_j}, t^{e_{j+1}}) \; .
  \end{split}
  \label{eq:prob2}
\end{align}

\section{Global search for local optima}
\label{sec:GO}

Interval arithmetic (IA) \cite{Moore2009} evaluations have the property that
all values that can be evaluated on a given domain are reliably contained in
the output of the corresponding interval evaluation.
To obtain global information on the function value a single function
evaluation in IA is required instead of multiple function evaluations at
several points.

For (compact) interval variable $[y]$ we use the notation
\begin{align*}
\left[y\right] = \left[\underline{y},\overline{y}\right]
= \left\{y\in \mathbb{R}^{n_y}\ |\ \underline{y}_i\leq y_i \leq \overline{y}_i
\,,\ i=1,\ldots,n_y\right\} \; ,
\end{align*}
with lower and upper bound $\underline{y}, \overline{y}\in \mathbb R^{n_y}$.
The united extension~(UE) of a function $h$ evaluated on $\left[y\right]$ is defined
as
\begin{align*}
h^\ast\left(\left[y\right]\right) = \left\{h(y) 
\in \mathbb{R}^{n_g}\ |\ y\in\left[y\right]\right\} \; .
\end{align*}
The UE for algebraic operators and elemental functions,
i.e., general power, general root, exponential, logarithmic, trigonometric and
hyperbolic functions, on compact domains are well known.
However, this does not apply to composite functions.
To enable IA of composite functions, the natural interval extension~(NIE) 
replaces all algebraic operators and elemental functions in the evaluation
procedure by their UE.
The evaluation of function $h$ on $\left[y\right]$ by the NIE yields
\begin{align*}
h\left(\left[y\right]\right) \supseteq h^\ast\left(\left[y\right]\right) \; .
\end{align*}
The superset states that the resulting interval can be an overestimation of the
UE.

Overestimation can occur if the underlying data format
(e.g., floating-point numbers)
cannot represent the exact bounds of a computed interval.
In this particular case the IA evaluation rounds towards negative or positive
infinity for a lower or upper bound, respectively.
Furthermore, overestimation can be caused by the dependency problem.
If a function evaluation uses a variable multiple times, IA does not take into
account that actual values taken from these intervals are equal.
The larger the intervals are the more significant the overestimation is.
Another challenge for IA are conditional branches that depend on interval
arguments.
Comparisons of intervals that intersect are ambiguous.
Splitting or multi-section \cite{Hansen2004} of the original domain and
evaluation in IA on subdomains might address this problem.
While the NIE converges linearly to the UE \cite{Moore2009}, so called
mean-value forms \cite{Moore1966} converge quadratically.
An alternative approach to obtain global information on the function value are
McCormick relaxations \cite{McCormick1976,Mitsos2009}.
These relaxations converge quadratically to the convex hull, from which the UE
can be determined.

To obtain guaranteed ranges for the derivatives on a given domain, the NIE can be
applied to derivative computations, e.g., by AD.
In \cite{Deussen2021} convergence of the interval methods applied to AD models is
shown and cases are investigated for which the NIE of the AD models yield the UE.
We will denote the interval gradient with
\begin{align*}
\left[\partial_y h\left(x,\left[y\right]\right)\right] \; ,
\end{align*}
and the interval Hessian with
\begin{align*}
\left[\partial^2_{yy}h\left(x,\left[y\right]\right)\right] \; .
\end{align*}
This information can be used to exclude that the necessary condition is
violated and to verify if the sufficient condition is fulfilled.

In \cite{Deussen2020}, it is shown how to sharpen the bounds of the
enclosure of the gradient by using McCormick relaxations instead of the NIE.
The nonsmooth McCormick relaxations are abs-factorable such that we can apply
piecewise linearization as suggested in\cite{Griewank2013}.
Thus, for computing the optima of the McCormick relaxations we use successive
piecewise linearization which is an optimization method that repeatedly
utilizes piecewise linearization models in a proximal point type method
\cite{Griewank2016}.

A divide and conquer algorithm can be utilized to detect all convex subdomains
$\hat B(\left[y\right])$ that potentially contain a local minimum of $h(x,y)$
\begin{align*}
\hat B(\left[y\right]) = \left\{\left[y\right]\ |\ 
  0 \in \left [\partial_y h\left(x,\left[y\right]\right)\right]\wedge
  0\prec\left[\partial^2_{yy} h\left(x,\left[y\right]\right)\right]\right\}\; .
\end{align*}
The existence of these boxes is shown by \Cref{prop:convexboxes}.
\begin{lemma} \label{lem:spectralnorm}
  The spectral norm of a symmetric positive definite matrix $A$ is its
  maximum eigenvalue.
\end{lemma}
\begin{proof}
  \begin{align*}
    \| A \| = \sqrt{ \lambda_{\max}(A^T A) }
    = \sqrt{ \lambda_{\max}(A^2) }
    = \sqrt{ \lambda_{\max}(A)^2 }
    = \lambda_{\max}(A) \; .
  \end{align*}
\end{proof}

\begin{proposition} \label{prop:convexboxes}
  For every local optimum $y^\diamond \in \argmin_y h(x, y)$ there is a closed box
  $$B(a,b)=\{ y\;|\; a_i\le y_i-y^\diamond_i \le b_i, i\in\{1,\dots,n_y\}\}$$
  with $y^\diamond \in \myint(B(a, b))$ such that
  $$\partial^2_{yy}h(x, y) \succcurlyeq \delta$$
  for $\delta > 0$ for all $y \in B(a, b)$.
\end{proposition}
\begin{proof}
  By \Cref{assump:Lipschitz} the Hessian is Lipschitz continuous on
  the box.
  Using \Cref{lem:spectralnorm} and the triangle inequality we get
  \begin{align*}
    \big| \lambda_\text{max}(\partial^2_{yy}h(x, y)) - \lambda_\text{max}(\partial^2_{yy}h(x, y^\diamond)) \big|
    &= \big| \| \partial^2_{yy}h(x, y) \| - \| \partial^2_{yy}h(x, y^\diamond) \| \big| \\
    &\le \| \partial^2_{yy}h(x, y) - \partial^2_{yy}h(x, y^\diamond) \| \\
    &\le L \| y - y^\diamond \|_2
  \end{align*}
  for some $L > 0$.
  By equivalence of norms in finite dimensional vector spaces we have
  \begin{equation*}
    \| y - y^\diamond \|_2 \le C \| y - y^\diamond \|_\infty \le C \max_{i \in \{1, \dots, n_y\}} \; b_i - a_i
  \end{equation*}
  for some $C > 0$.
  It follows that for any box $B(a, b)$ we can find a bound $Z > 0$ such
  that
  \begin{equation*}
    \lambda_\text{max}(\partial^2_{yy}h(x, y)) \le \lambda_\text{max}(\partial^2_{yy}h(x, y^\diamond)) + Y = Z
  \end{equation*}
  for all $y \in B(a, b)$.
  Let $I$ be the $n_y \times n_y$ identity matrix.
  The shifted matrix
  $$\partial^2_{yy}h(x, y) - Z I$$
  has the dominant eigenvalue $\lambda_\text{min}(\partial^2_{yy}h(x, y)) - Z$
  for all $y \in B(a, b)$.
  Using \Cref{lem:spectralnorm}, the triangle inequality, strong convexity
  of the Hessian and Lipschitz continuity (\Cref{assump:Lipschitz}) we get
  \begin{align*}
    \lambda_\text{min}(\partial^2_{yy}h(x, y))
    &= Z - \| \partial^2_{yy}h(x, y) - Z I \| \\
    &= Z - \| \partial^2_{yy}h(x, y^\diamond) 
    + (\partial^2_{yy}h(x, y) - \partial^2_{yy}h(x, y^\diamond)) - Z I \| \\
    &\ge Z - \| \partial^2_{yy}h(x, y^\diamond) - Z I \|
     - \| \partial^2_{yy}h(x, y) - \partial^2_{yy}h(x, y^\diamond) \| \\
    &\ge Z - \| \partial^2_{yy}h(x, y^\diamond) - Z I \| - L \| y - y^\diamond \| \\
    &= \lambda_\text{min}(\partial^2_{yy}h(x, y^\diamond)) - L \| y - y^\diamond \| \\
    &\ge \mu - L \| y - y^\diamond \| \\
    &\ge \mu - L C \max_{i \in \{1, \dots, n_y\}} \; b_i - a_i
  \end{align*}
  for all $y \in B(a, b)$.
  We choose $b_i > 0$ and $a_i < 0$ then $y^\diamond \in \myint(B(a, b))$
  and we choose them such that
  \begin{equation*}
    0 < b_i - a_i \le \frac{\mu - \delta}{L C}
  \end{equation*}
  then
  $\mu - L C (b_i - a_i) \ge \delta$
  for all $i \in \{1, \dots, n_y\}$.
\end{proof}

The divide and conquer algorithm recursively refines the domain until
a subdomain can either be eliminated due to violation of the optimality
condition or be returned due to approval of these conditions.
The union of the returned and the eliminated subdomains should cover
the original domain.
The processing of a subdomain $B$ consists of the following main components:
\begin{itemize}
\item \emph{Eliminating}: Discard $B$, if
$$0\notin \left[\partial_y h\left(x,\left[y\right]\right)\right] \vee
0\nprec\left[\partial^2_{yy} h\left(x,\left[y\right]\right)\right]\; .$$
\item \emph{Terminating}: Append $\hat{B}([y])$ to $\mathcal B(t)$, if
$$0\in \left[\partial_y h\left(x,\left[y\right]\right)\right] \wedge
0\prec\left[\partial^2_{yy} h\left(x,\left[y\right]\right)\right]\; .$$
\item \emph{Branching}: Split the current domain $B$ into subdomains $B_j$
with $B = \bigcup_j B_j$.
\end{itemize}
Performing local searches on the convex subdomains stored in
$\mathcal B(t)$ results in the set $\mathcal S(t)$ that contains all local
minima $y^\diamond$.
Some of these subdomains might not contain a local minimum due to the
already mentioned overestimation of IA.
In that case the local search would find the minimum on the bound of
the subdomain and would not include this into the set $\mathcal S(t)$.

To reduce the computational effort of the complete search algorithm, it might
be desirable to implement a bounding step into the algorithm as it is done by
conventional branch-and-bound algorithms as implemented in DGO solver, e.g.,
MAiNGO \cite{Bongartz2018} or BARON \cite{Khajavirad2018}.
Instead of finding all local optima $y^\diamond$, one would focus on those
local optima that are close to the global optimum $y^\ast$.
These local optima would fulfill
$$h(x,y^\diamond) \leq h(x,y^\ast) + \alpha\; ,$$
for $\alpha>0$.
That would lead to the additional steps in the processing of subdomain $B$:
\begin{itemize}
\item \emph{Bounding}: Compute a guaranteed lower bound $\underline{h}$ of a
convex relaxation of the objective function on $B$.
\item \emph{Eliminating}: Discard $B$, if 
$$\underline{h} > \overline{h}^\ast + \alpha\; .$$
\item \emph{Pruning}: evaluate the function at any feasible point $y\in B$, e.g.,
the midpoint $y^\dagger$ of $B$, and update the current best solution
$\overline{h}^\ast$, if $h(x, y^\dagger)<\overline{h}^\ast$.
\end{itemize}

\section{Numerical simulation}
\label{sec:NS}

In this section, we present how to solve \cref{eq:prob2} numerically with
a temporal discretization for the numerical integrator
and how to detect and locate events at which the global optimizer switches.

\subsection{Time stepping}
We select an implicit scheme for the time stepping of the simulation
of the dynamic system in order to be stable and a higher-order
scheme in order to be more efficient in terms of step size.
The trapezoidal rule, a second-order convergent Runge-Kutta method,
computes
$$x(t_{k+1}) = x(t_{k}) + \frac{\Delta t_k}2 \left(f\left(x(t_{k+1}),y(t_{k+1})\right)+f\left(x(t_{k}),y(t_{k})\right)\right)\; ,$$
for discrete time steps $t_{k}$ and $t_{k+1}$,
with $\Delta t_k = t_{k+1}-t_{k}$ such that we obtain the discretized system
\begin{align}
\begin{split}
  x^j(t^{e_j}) &= \begin{cases}
    x_0 & \text{if } j=0\\
    x^{j-1}(t^{e_j}) & \text{otherwise}
  \end{cases} \\
  x^j(t_{k+1}) &= x^j(t_{k}) + \frac{\Delta t_k}2 \left(f\left(x^j(t_{k+1}),y^j(t_{k+1})\right)+f\left(x^j(t_{k}),y^j(t_{k})\right)\right)\\
  0 &= \partial_y h(x^j(t_{k+1}), y^j(t_{k+1})) \quad t_k,t_{k+1} \in (t^{e_j}, t^{e_{j+1}}) \; .
  \end{split}
  \label{eq:prob3}
\end{align}
The system in \cref{eq:prob3} can be solved by a linear solver, e.g., with LU
decomposition.
Since the events $t^{e_j}$, $j\in\{1,\ldots,K\}$, are not known apriori, we need a
need a mechanism for event detection.

\subsection{Event detection}

Let us assume we track the dynamic set of multiple strong local optima
\begin{equation*}
  \mathcal S(t) = \{ y^1(t), \dots, y^S(t) \} \; .
\end{equation*}
No local optimum emerges or vanishes between $t_k$ and $t_{k+1}$, such that we
have $|\mathcal S(t_k)| = |\mathcal S(t_{k+1})| = S$.
Furthermore, we assume that there is an event at $t^e$.
The unique global optimum for $t_k \le t < t^e$ is $y^1(t)$.
We have $$h(y^1(t)) < h(y^s(t))$$ for all
$s \in \{ 1, \dots, S \} \setminus \{ 1 \}$, $t \in [t_k, t^e)$.
An event occurs at $t^e$ when without loss of generality
$h(y^2(t^e)) = h(y^1(t^e))$.
Due to transversality there exists a $t_{k+1} > t^e$ such that
$$h(y^2(t)) < h(y^s(t))$$ for all
$s \in \{ 1, \dots, S \} \setminus \{ 2 \}$, $t \in (t^e, t_{k+1}]$.
Note, that we keep the superscripts of the local optima from $t_k$ for
$t\in[t_k,t_{k+1}]$.

The question of event detection in the discrete time setting is the following:
Looking only at randomly ordered list representations of $\mathcal S(t_k)$ and
$\mathcal S(t_{k+1})$ that are obtained by a (global) optimization algorithm,
has a switch of the global optimum happened in the meantime?

We are assuming that only finitely many isolated switches happen.
So we can always choose $\Delta t_k = t_{k+1} - t_k$ small enough such that at
most one switch happens per time step.
The problem is not that a back and forth switch might happen
between $t_k$ and $t_{k+1}$ that we do not observe (because the optimum at
$t_{k+1}$ is again the same as $t_k$).
The problem is to distinguish the elements of the sets $\mathcal S(t_k)$ and
$\mathcal S(t_{k+1})$ in such a way that it is clear if one switch or no switch
has happened.

The elements of the sets $\mathcal S(t_k)$ and $\mathcal S(t_{k+1})$ have no
intrinsic order.
So for example $y^1(t_k)$ could be the first element of a list
representation of $\mathcal S(t_k)$ but $y^1(t_{k+1})$ the second element of
a list representation of $\mathcal S(t_{k+1})$.
The task of event detection comes down distinguishing elements in the list
representation.
Because then we can check whether $h(y^s(t_{k+1})) < h(y^1(t_{k+1}))$ for some
$s \not= 1$.

Let $\hat y^1_k = y^1(t_k)$ be the first element of $\mathcal S(t_k)$.
We want to determine what elements $\hat y^s$ for $s \in \{ 1, \dots, S \}$
could potentially be $\hat y^s = y^1(t_{k+1})$.
In order to do this we have to consider that $y^1(t) = y^1(x(t))$ is a locally
Lipschitz continuous implicit function of $x$.

\begin{proposition}
  Given $y^1(t_k)$ we can bound the values of $y^1(t_k + \Delta t_k)$ by
  a Lipschitz constant
  \begin{equation*}
    \| y^1(t_k + \Delta t_k) - y^1(t_k) \| \le L \| \Delta t_k \|
  \end{equation*}
  for some $L > 0$.
\end{proposition}
\begin{proof}
  We have $x(t)$ is absolutely continuous due to Carathéodory's existence
  theorem \cite{Hale1980}.
  We are considering a compact time interval so absolute continuity implies
  Lipschitz continuity
  \begin{equation*}
    \| x(t_k + \Delta t_k) - x(t_k) \| \le L_1 \| \Delta t_k \| \; .
  \end{equation*}
  The implicit function $y(x)$ is Lipschitz continuous due to strong local
  convexity and implicit function theorem by mean value theorem
  \begin{equation*}
    y(x + \Delta x) - y(x) = \partial_x y(\xi) \Delta x \le \| \partial_x y(\xi) \| \Delta x
  \end{equation*}
  for some $\xi \in \{ x + \lambda \Delta x \;|\; \lambda \in [0, 1] \}$.
  With Lipschitz continuity of the implicit function $y(x)$ we get
  \begin{align*}
    \| y^1(t_k + \Delta t_k) - y^1(t_k) \| 
    &= \| y^1(x(t_k + \Delta t_k)) - y^1(x(t_k)) \| \\
    &= \| y^1(x(t_k) + x(t_k + \Delta t_k) - x(t_k)) - y^1(x(t_k)) \| \\
    &\le L_2 \| x(t_k + \Delta t_k) - x(t_k)) \| \\
    &\le L_2 L_1 \| \Delta t_k \| \; .
  \end{align*}
\end{proof}
The same idea holds in the discrete time setting when allowing for some error.
This is obvious for example when using an explicit Euler step where the change
between $x(t_{k+1})$ and $x(t_k)$ is directly given by the dynamics.
Instead of the real Lipschitz constant in $t$ we just use the derivative of
that time step $(y^1)'(t_k)$.

The event detection can be implemented into an integrator by verifying
\begin{align}
\|y^\ast(t_k) - y^\ast(t_{k+1})\| \leq \|\Delta t_k\| \cdot \|(y^\ast)'(t_k)\|\; ,
\label{eq:event_detection}
\end{align}
in every time step.
An event is detected if \cref{eq:event_detection} is violated.
In that case, it remains to find the event location.

\subsection{Event location}

We mix the event location procedure into the time stepping procedure via the
Mannshardt approach \cite{Mannshardt1978} that is compatible with our
assumptions.

Finding the event times $t^{e_j}$ involves numerically solving for the event
function roots.
We use Newton's method to solve $0 = H(\tilde x(t^{e_j}))$ for
$t^{e_j}$ where $\tilde x(t^{e_j})$ is a local numerical approximation of
the solution trajectory.
Differentiation of $\tilde x(t^{e_j})$ with respect to $t^{e_j}$ is
straight-forward so we need the gradient of the event function with respect
to $x$
\begin{align*}
  \partial_x H(x)  &= \partial_x h(x, y^1(x))
  + \partial_y h(x, y^1(x)) \partial_x y^1(x)\\
  &- \partial_x h(x, y^2(x))
  - \partial_y h(x, y^2(x)) \partial_x y^2(x) \; .
\end{align*}
The individual $\partial_x y^i(x)$ are computed by implicit function theorem
as seen above.
Note, that $y^2(t_{k+1})$ can be obtained directly from the DGO solver since
it is the global optimum, while $y^1(t_{k+1})$ needs to be identified in the
set of local optima $\mathcal S(t_{k+1})$.
This can be achieved by local optimization, i.e., solving
$$\partial_x h(x(t_{k+1},y^1(t_{k+1})) = 0\; ,$$
with initial guess
$$y^1(t_{k+1})= y^1(t_{k})+ \Delta t_k \cdot (y^\ast)'(t_k)\; .$$

\section{Numerical experiments}
\label{sec:NE}

We investigate two example DAEOs with global optima discontinuous in time.
The first example is a very basic example with two local optima for which an
analytical solution can be easily derived.
The second example has multiple local optima that are emerging and vanishing
over time.
We will compare numerical simulations with and without event detection.
Furthermore, we set the time step for the numerical integration to
$\Delta t=0.02$ and use the trapezoidal rule as explained above.
For the DGO part of the simulation, we use the solver described in
\cite{Deussen2019}.

\begin{figure*}[tb]
  \centering
\begin{tikzpicture}
[declare function={t(\p) = -ln(\p)/3.0;},
  declare function={
    xt(\x)= (\x<=-ln(0.5)/3.0) * exp(-3.0*\x)   +
     (\x>-ln(0.5)/3.0) * exp(-\x+2.0/3.0*ln(0.5));
  },
  declare function={
    yt(\x)= (\x<=-ln(0.5)/3.0) * 1   +
     (\x>-ln(0.5)/3.0) * -1;
  },
  declare function={
    dpx(\x)= (\x>-ln(0.5)/3.0) * exp(-\x+2.0/3.0*ln(0.5))*2.0/3.0/0.5;
  },
  declare function={
    h(\x,\t)= (1.0-\x*\x)*(1.0-\x*\x) - (xt(\t)-0.5)*sin(deg(\x*pi/2.0));
  },
  ]
\begin{axis}[width=0.45\linewidth, height=0.38\linewidth,
xmin=-1.4, xmax=1.4, xlabel=$y$, restrict x to domain=-1.4:1.4,
ymin=0, ymax=1.0, ylabel=$t$, restrict y to domain=0:1.0,
zmin=-1.0, zmax=1.2, zlabel={$h(x(t),y(t))$}, restrict z to domain=-1.0:1.2,
y tick label style={yshift=0.3cm},
extra y ticks = {0.23104906}, extra y tick labels = {$t^e$},
extra y tick style={grid=major,ticklabel pos=upper},
y label style={yshift=0.7cm},
z label style={xshift=0.5cm, yshift=-0.35cm},
view={9}{22}
]
  \addplot3[mesh, color=black, domain=-1.4:1.4, y domain=0:1.0, samples=30, samples y=20, thin] (x, y, {h(x,y)});
  \addplot3[color=Color4, densely dashed, domain=0:0.23104906, samples y=1, smooth, ultra thick] (-1.0, x, {h(-1,x)});
  \addplot3[color=Color2, densely dashed, domain=0.23104906:1, samples y=1, smooth, ultra thick] (-1.0, x, {h(-1,x)});
\addplot3[color=Color0, domain=0:0.23104906, samples=2, samples y=1, ultra thick] (1,x,-1);
\addplot3[color=Color0, domain=0.23104906:1, samples=2, samples y=1, ultra thick] (-1,x,-1);
  \addplot3[color=Color2, domain=0:0.23104906, samples y=1, smooth, ultra thick] (1.0, x, {h(1,x)});
  \addplot3[color=Color4, domain=0.23104906:1, samples y=1, smooth, ultra thick] (1.0, x, {h(1,x)});
\end{axis}
\begin{scope}[xshift=0.5\linewidth]
\begin{axis}[width=0.48\linewidth, height=0.38\linewidth,
xmin=0, xmax=1, xlabel=$t$, restrict x to domain=0.0:1.0,
ymin=-1.9, ymax=1.9, restrict y to domain=-1.7:1.7,
extra x ticks = {0.23104906}, extra x tick labels = {$t^e$},
extra x tick style={grid=major,ticklabel pos=upper},
legend style={font=\small, fill=none,draw=none}, legend columns=2,
x tick label style={yshift=-0.5ex},
y tick label style={xshift=-0.5ex}
]
  \addplot[color=Color2, domain=0:0.23104906, samples=20, smooth, very thick] (x, {h(1,x)});
  \addlegendentry{}
  \addplot[color=Color4, domain=0.23104906:1, samples=60, smooth, very thick] (x, {h(1,x)});
  \addlegendentry{$h(x(t),y^1(t))$}
  \addplot[color=Color4, domain=0:0.23104906, samples=20, densely dashed, smooth, very thick] (x, {h(-1,x)});
  \addlegendentry{}
  \addplot[color=Color2, domain=0.23104906:1, samples=60, densely dashed, smooth, very thick] (x, {h(-1,x)});
  \addlegendentry{$h(x(t),y^2(t))$}
  \addplot[fill=none,draw=none,opacity=0] coordinates {(0.2,0.2)};
  \addlegendentry{};
  \addplot[color=Color0, domain=0:0.23104906, samples=2, thick] {1};
  \addplot[color=Color0, domain=0.23104906:1, samples=2, thick] {-1};
  \addlegendentry{$y(t)$};
\end{axis}
\end{scope}
\end{tikzpicture}
\caption{The left image shows a surface plot of $h(x(t),y(t))$ for $t\in[0,1]$.
The position $y(t)$ of the global optimum which lies at $y^1=1$ for $t<t^e$ and
at $y^2=-1$ for $t>t^e$ is marked by the red line while the green lines indicate
its value.}
\label{fig:ex1_h}
\end{figure*}
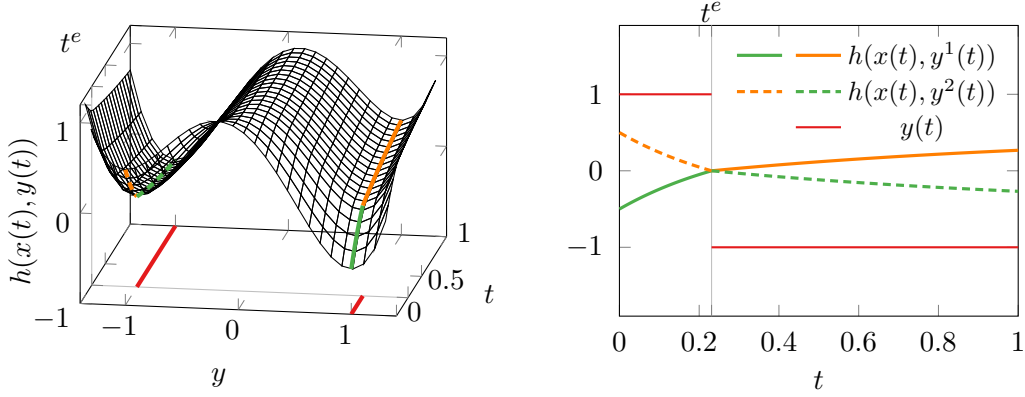
\begin{example} \label{ex:1}
As a first example, we consider
  \begin{align*}
    x(0) &= 1\\
    x'(t) &=  -\left(2+y(t)\right)x(t) \\
    \{ y(t) \} &= \argmin_y \; h(x(t),y(t)) \; ,
  \end{align*}
with $T=1$ and 
\[h(x(t),y(t)) = \left(1 -(y(t))^2\right)^2 - \left(x(t)-\frac12\right)\sin\left(y(t)\frac\pi2\right)\; .\]
Function $h(x(t),y(t))$ has two local optima at time-independent positions
$y^1(t)=1$ and $y^2(t)=-1$, i.e., $\partial_t y^1(t) = 0$
and $\partial_t y^2(t) = 0$.
A surface plot of $h(x(t),y(t))$ is shown in \Cref{fig:ex1_h}.
Finding the root of the event function, i.e.,
\[H(x(t))=h(x(t),y^1(t))-h(x(t),y^2(t))=0\; ,\]
yields event location at $x(t)=0.5$.
The solution of the DAEO is
\[x(t)=\exp(-3t)\;,\ t\in[0,t^e]\;,\]
for the first phase
\begin{align*}
  x(0) &= 1\\
  x'(t) &= -3x(t)\; .
\end{align*}
Thus, the event location is at $t^e=-\log(0.5)/3$.
The second phase is described by
\begin{align*}
  x(t^e) &= 0.5\\
  x'(t) &= -x(t)\; ,
\end{align*}
with solution
\[x(t)=\exp\left(-t+\frac23\log(0.5)\right)\;,\ t\in[t^e,1]\; .\]

\end{example}
\begin{figure*}[htb]
  \centering
\begin{tikzpicture}[
  declare function={
    xt(\x)= (\x<=-ln(0.5)/3.0) * exp(-3.0*\x)   +
     (\x>-ln(0.5)/3.0) * exp(-\x+2.0/3.0*ln(0.5));
  },
  declare function={
    yt(\x)= (\x<=-ln(0.5)/3.0) * 1   +
     (\x>-ln(0.5)/3.0) * -1;
  },
  declare function={
    dpx(\x)= (\x>-ln(0.5)/3.0) * exp(-\x+2.0/3.0*ln(0.5))*2.0/3.0/0.5;
  }
]
\begin{axis}[width=0.4\linewidth, height=0.35\linewidth,
xmin=0, xmax=1, xlabel=$t$,
ymin=0, ymax=1,
extra x ticks = {0.23104906}, extra x tick labels = {$t^e$},
extra x tick style={grid=major,ticklabel pos=upper},
legend style={font=\small, fill=none,draw=none}, legend columns=1,
x tick label style={yshift=-0.5ex},
y tick label style={xshift=-0.5ex}
]
  \addplot[color=Color1, domain=0:1, samples=500, ultra thick] {xt(x)};
  \addlegendentry{$x(t)$};
  \addplot[color=Color4, mark=+, only marks, mark repeat=4,mark size=3.6pt, very thick] table [col sep=comma] {data/ex1_noeh.txt};
  \addlegendentry{$\hat x(t)^-$}
    \addplot[color=black, mark=x, only marks, mark repeat=4,mark size=3.6pt, very thick] table [col sep=comma] {data/ex1_eh.txt};
  \addlegendentry{$\hat x(t)^+$}
\end{axis}
\begin{scope}[xshift=0.5\linewidth]
\begin{axis}[width=0.5\linewidth, height=0.35\linewidth,
xmin=1e-6, xmax=0.1, xmode=log, xlabel=$\Delta t$,
xminorticks=false,
ymin=1e-13, ymax=1, ymode=log, ylabel=$\Delta t \|\hat \cdot- \cdot\|_1$,
legend style={font=\small, at={(1.02,-0.02)},
anchor=south east,fill=none,draw=none}, legend columns=1,
x tick label style={yshift=-0.5ex},
y tick label style={xshift=-2.5ex}
]
  \addplot[color=black, domain=1e-6:1, samples=2, dashdotted, thick] {x};
  \addlegendentry{$\mathcal O(\Delta t)$};
  \addplot[color=Color4, mark=+, mark size=3.6pt, very thick] table [x=dt, y=x1] {data/conv_noeh.txt};
  \addlegendentry{$x(t)^-$}
  \addplot[color=black, domain=1e-6:1, samples=2, densely dashed, thick] {x^2};
  \addlegendentry{$\mathcal O(\Delta t^2)$};
  \addplot[color=black, mark=x, mark size=3.6pt, very thick] table [x=dt, y=x1] {data/conv_eh.txt};
  \addlegendentry{$x(t)^+$}
\end{axis}
\end{scope}
\end{tikzpicture}
\caption{Analytical (line) and numerical (marks) results for differential
variable $x(t)$ of \Cref{ex:1} on the left.
The differential equation is solved by trapezoidal rule with $\Delta t=0.02$
without (superscript $-$) and with (superscript $+$) explicit treatment of events.
The convergence behavior of the two methods is shown on the right.}
\label{fig:ex1_x}
\end{figure*}
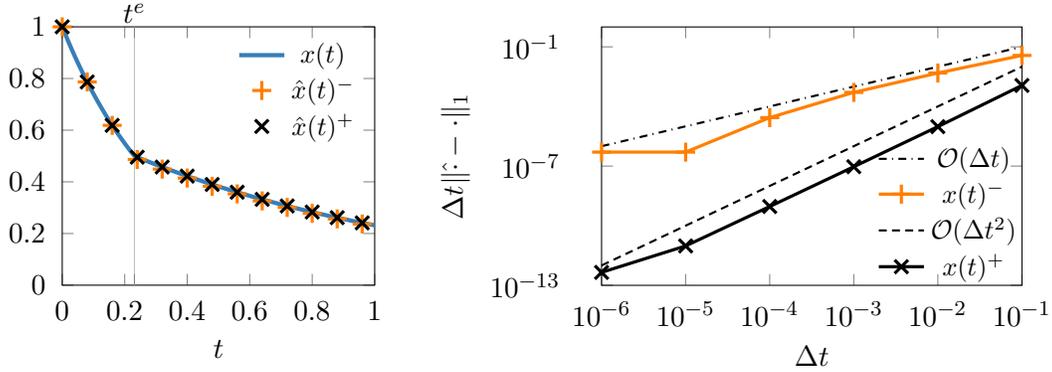

On the left side of \Cref{fig:ex1_x}, the results of the differential
variable $x(t)$ are shown.
The blue line represents the analytical solution, while the orange circles
mark the numerical simulation without event detection and the black crosses
mark the one with event detection (and location).
It becomes visible that the differential variable is nonsmooth at the event.
On the right side of \Cref{fig:ex1_x} the convergence behavior of the numerical
simulations with (black, crosses) and without (orange, circles) explicit
treatment of the event is plotted for decreasing time steps.
It can be seen that the version with explicit treatment converges quadratically
to the analytical solution, while the version without explicit treatment only
converges linearly.
Computing the event location and treating this explicitly in the simulation
is only possible due to the tracking and detection of switches in the global
optimum.

\begin{figure*}[htb]
  \centering
\begin{tikzpicture}
[declare function={t(\p) = -ln(\p)/3.0;},
  declare function={
    xt(\x)= (\x<=-ln(0.5)/3.0) * exp(-3.0*\x)   +
     (\x>-ln(0.5)/3.0) * exp(-\x+2.0/3.0*ln(0.5));
  },
  declare function={
    yt(\x)= (\x<=-ln(0.5)/3.0) * 1   +
     (\x>-ln(0.5)/3.0) * -1;
  },
  declare function={
    dpx(\x)= (\x>-ln(0.5)/3.0) * exp(-\x+2.0/3.0*ln(0.5))*2.0/3.0/0.5;
  },
  declare function={
    h(\x,\t)= (\x-1)*(\x-1) + sin(deg(4*\x))+1;
  },
  ]
\begin{axis}[width=0.40\linewidth, height=0.35\linewidth,
xmin=0, xmax=6, xlabel=$y$, restrict x to domain=0:6,
ymin=0, ymax=2, ylabel=$t$, restrict y to domain=0:2,
zmin=-1, zmax=20, zlabel={$h(x(t),y(t))$}, restrict z to domain=-1:30,
extra y ticks = {0.589331037698447, 1.16042284547823, 1.5238123362291, 1.79029462182903},
extra y tick labels = {$t^{e_1}$, $t^{e_2}$, $t^{e_3}$, $t^{e_4}$},
extra y tick style={grid=major,ticklabel pos=upper,ticklabel style={xshift=-0.15cm,yshift=-0.07cm}},
x label style={yshift=0.3cm},
y label style={yshift=0.5cm},
view={-21}{32},
colormap={custom}{color(0)=(black!100) color(1)=(black!20)}
]
  \addplot3[mesh, point meta=\thisrow{t}, shader=interp] table [col sep=comma, x=y, y=t, z=h] {data/ex2.txt};
  \addplot3[color=Color0, very thick] table[col sep=comma, x=y, y=t, z=h] {data/ex2_eh.txt};
\end{axis}
\begin{scope}[xshift=0.4\linewidth]
\begin{axis}[width=0.55\linewidth, height=0.35\linewidth,
xmin=0, xmax=2, xlabel=$t$, restrict x to domain=0:2,
ymin=0, ymax=8, restrict y to domain=0:8,
extra x ticks = {0.589331037698447, 1.16042284547823, 1.5238123362291, 1.79029462182903},
extra x tick labels = {$t^{e_1}$, $t^{e_2}$, $t^{e_3}$, $t^{e_4}$},
extra x tick style={grid=major,ticklabel pos=upper},
legend style={font=\small, fill=white,draw}, legend columns=2,
legend pos=north west,
x tick label style={yshift=-0.5ex},
y tick label style={xshift=-0.5ex}
]

  \addplot[color=Color1, very thick, dashed] table [col sep=comma] {data/ex2_noeh.txt};
  \addlegendentry{$\hat x(t)^-$}
  \addplot[color=Color1, very thick] table [col sep=comma ] {data/ex2_eh.txt};
  \addlegendentry{$\hat x(t)^+$}
    \addplot[color=Color0, very thick, dashed] table [col sep=comma, y=y] {data/ex2_noeh.txt};
  \addlegendentry{$\hat y(t)^-$}
    \addplot[color=Color0, very thick] table [col sep=comma, y=y] {data/ex2_eh.txt};
  \addlegendentry{$\hat y(t)^+$}
    \addplot[color=Color2, very thick, dashed] table [col sep=comma, y=h] {data/ex2_noeh.txt};
  \addlegendentry{$\hat h^-$}
    \addplot[color=Color2, very thick] table [col sep=comma, y=h] {data/ex2_eh.txt};
  \addlegendentry{$\hat h^+$}

\end{axis}
\end{scope}
\end{tikzpicture}
\caption{Surface plot of the 1D Griewank inspired function from example \Cref{ex:2}
with global optimizer (red line) is shown on the left. The right side shows
differential variable $x(t)$ (blue), global optimizer $y(t)$ (red) and value of
the global optimum (green) for the version with (solid) and without (dashed)
explicit treatment of events.}
\label{fig:ex2_h}
\end{figure*}
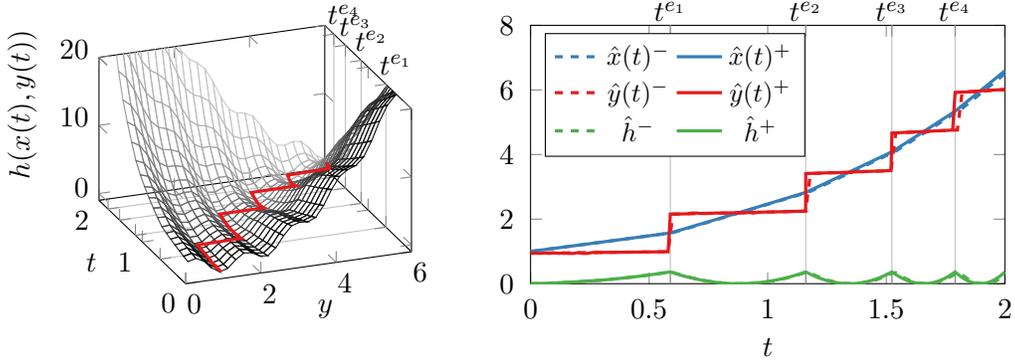

\begin{example} \label{ex:2}
The second DAEO considers a 1D Griewank inspired function \cite{Griewank1981}
as embedded optimization problem.
The DAEO is described by
  \begin{align*}
    x(0) &= 1 \\
    x'(t) &= y(t) \\
    \{ y(t) \} &= \argmin_y \; (y(t) - x(t))^2 + \sin(C y(t))\; ,
  \end{align*}
with $C=5$ and $T=2$.
The objective function has multiple optima that are emerging and vanishing over
time.
A surface plot of the function with the global optimizer $y(t)$ (red line)
is given on the left side of \Cref{fig:ex2_h}.
It can be seen that there are four events on the time interval $t\in[0,2]$.
The differentiable variable $x(t)$ (blue), the global optimizer $y(t)$ (red)
and the global optimum $h(x(t),y(t))$ (green) are represented as lines in the
right graph of \Cref{fig:ex2_h}.
The numerical results of the version without (dashed) explicit treatment of
the event differ slightly from the result with explicit treatment.
This is a result of the error propagation due to the behavior of the dynamic
system.
\end{example}

\section{Conclusions and further work}
\label{sec:C}
We introduced event detection and event location for a jumping global optimizer.
The explicit treatment of events and its implementation into the numerical
integrator yield a second-order convergent method for DAEOs in the presence of
a jumping global optimizer.
The second-order integrator enables the computation of discrete tangent and
adjoint sensitivities of the dynamic system with respect to some parameters,
which is crucial for solving optimal control problems of DAEOs.
These sensitivities can now be obtained by AD methods.

Due to the high computational cost that comes along with the application of DGO
in every time step, we aim for tracking relevant local optima in time instead.
The global search is only required for getting a list of local optima and for
recognizing that a new local optimum emerged during a time step.
Vanishing optima should not be a problem with this approach.

\section*{Acknowledgement}
This work was supported by the German Research Foundation (DFG) under grant number
NA487/8-2.

%

\end{document}